\documentclass[11pt]{amsart}

\usepackage{amsmath, amssymb, amscd, cancel, graphicx, soul, stmaryrd}
\usepackage[T1]{fontenc}
\usepackage{mathdots}

\usepackage[dvipsnames,usenames]{color}
\usepackage[colorlinks=true, urlcolor=myblue, linkcolor=myblue, citecolor=myblue]{hyperref}
\usepackage{float}
\usepackage{amsmath, amsthm, amssymb}
\usepackage{amsfonts}
\usepackage{enumerate}
\usepackage{paralist}
\usepackage{xypic}
\usepackage{subfigure}
\usepackage{hyperref}
\usepackage{color}
\usepackage{marginnote}
\usepackage{mathrsfs}

\definecolor{myblue}{RGB}{83,118,182}

\hypersetup{
linkbordercolor={1 0 0}, 
citebordercolor={0 1 0} 
}
\newtheorem{thm}{Theorem}[section]

\newtheorem{cor}[thm]{Corollary}

\newtheorem{prop}[thm]{Proposition}

\newtheorem{defin}[thm]{Definition}
\newtheorem{rmk}[thm]{Remark}
\newtheorem{ex}[thm]{Example}

\def\bQ{{\mathbb Q}}

\def\bZ{{\mathbb Z}}

\def\zsh{{g_{sh}^0(K)}}
\def\gs{{g_4(K)}}
\def\sC{{\mathscr{C}}}

\begin{document}
\bibliographystyle{alpha}

\title[]%
{Shake genus and slice genus}

\author{Lisa Piccirillo}
\address{Department of Mathematics, University of Texas, Austin, TX 78712}
\email{lpiccirillo@math.utexas.edu}

\maketitle

\begin{abstract}
An important difference between high dimensional smooth manifolds and smooth 4-manifolds that in a 4-manifold it is not always possible to represent every middle dimensional homology class with a smoothly embedded sphere. This is true even among the simplest 4-manifolds: $X_0(K)$ obtained by attaching an $0$-framed 2-handle to the 4-ball along a knot $K$ in $S^3$. The $0$-shake genus of $K$ records the minimal genus among all smooth embedded surfaces representing a generator of the second homology of $X_0(K)$ and is clearly bounded above by the slice genus of $K$. We prove that slice genus is not an invariant of $X_0(K)$, and thereby provide infinitely many examples of knots with $0$-shake genus strictly less than slice genus. This resolves Problem 1.41 of \cite{Kir97}. As corollaries we show that Rasmussen's $s$ invariant is not a $0$-trace invariant and we give examples, via the satellite operation, of bijective maps on the smooth concordance group which fix the identity but do not preserve slice genus.  These corollaries resolve some questions from \cite{4MKC16}.
\end{abstract}

\section{Introduction}
One of the key differences between smooth 4-manifolds and higher dimensional smooth manifolds is the ability to represent any middle dimensional homology class with a smoothly embedded sphere. For 4-manifolds this is not always possible even in the simplest case: four-manifolds $X_n(K)$ obtained by attaching an $n$-framed 2-handle to the 4-ball along a knot $K$ in $S^3$. We call such a manifold the $n$-trace of $K$. The $n$-shake genus of $K$, denoted $g_{sh}^n(K)$, measures this failure to find a sphere representative by recording the minimal genus among smooth embedded generators of the second homology of $X_n(K)$. 

Recall that the slice genus of $K$, denoted $g_4(K)$, is defined to be the minimal genus of a smooth properly embedded surface $\Sigma\hookrightarrow B^4$ such that $\partial\Sigma=K\subset S^3$. When we attach a 2-handle to $B^4$ along $K$, any such $\Sigma$ can be capped off to a closed surface $\widehat{\Sigma}$ of the same genus. So we see that for all integers $n$ and knots $K$ the $n$-shake genus is bounded above by the slice genus. Since $\widehat{\Sigma}$ is embedded in a restrictive manner ($\widehat{\Sigma}$ intersects the cocore of the handle in one point) one might expect that the $n$-shake genus can be strictly less than the slice genus. Indeed for $n\neq 0$ such examples are well-known \cite{Akb77,Lic79,Akb93,AJOT13,CR16}. All of these examples rely on the same proof technique: produce two knots $K$ and $K'$ with $X_n(K)$ diffeomorphic to $X_n(K')$, then show that $g_4(K)\neq g_4(K')$. This paper concerns the case $n=0$.

There are a few issues with using the $n\neq 0$ proof outline to show that there exist $K$ such that $\zsh <\gs$. A longstanding conjecture of Akbulut and Kirby, Problem 1.19 of \cite{Kir97}, held that if the 3-manifolds $S^3_0(K)$ and $S^3_0(K')$ obtained by 0-framed Dehn surgery on knots $K$ and $K'$ are homeomorphic, then $K$ and $K'$ are (smoothly) concordant.  Since $S^3_0(K)$ arises naturally as $\partial X_0(K)$, any attempt to use the $n\neq0$ argument to show  there exist $K$ such that $\zsh <\gs$ must disprove this conjecture along the way. It is a priori possible to show that there exist knots $K$ with $\zsh <\gs$ without exhibiting a counterexample to the Akbulut Kirby conjecture but to do so one would need to demonstrate that there exists a knot $K$ with $\zsh <\gs$ and so that for any smooth embedded minimal genus surface $\Gamma$ generating $H_2(X_0(K))$ there is no description of $X_0(K)$ consisting of only a $0$ and $2$ handle in which $\Gamma$ can be isotoped to meet the cocore of the $2$ handle transversely in a point. The existence of such a knot seems to be a subtle problem; it is not resolved here. 

In 2015 Yasui disproved the Akbulut-Kirby conjecture \cite{Yas15a}, and it was later demonstrated by A. N. Miller and the author that there exist non-concordant knots with diffeomorphic $0$-traces \cite{MP17}. These results give some evidence that that it might in fact be possible to show that there exist $K$ with $\zsh <\gs$ by producing knots $K$ and $K'$ with $X_0(K)$ diffeomorphic to $X_0(K')$ and $g_4(K)\neq g_4(K')$.

Some concerns with that outline remain; first, there is a brief classical argument, see \cite{KM78}, showing that if a knot $K$ shares a 0-trace with a slice knot $K'$ then $K$ is slice. As such, the standard outline cannot be used to show that there exist 0-shake slice knots which are not slice; the existence of such knots remains open. Second, the $0$-surgery of a knot, and hence the $0$-trace, determines fundamental slice genus bounds such as the Tristram-Levine signatures, Casson-Gordon signatures, and signature invariants associated to the filtration of \cite{COT03}. In fact  there was only one smooth concordance invariant known not to be a $0$-trace invariant, and that invariant does not give slice genus bounds, see \cite{MP17}. Nevertheless, we have

\begin{thm} \label{Thm:genusgap}There exist infinitely many pairs of knots $K$ and $K'$ such that $X_0(K)$ is diffeomorphic to $X_0(K')$ and $\gs\neq g_4(K')$.
\end{thm}

\begin{cor}\label{cor:shakenotslice}
There exist infinitely many knots $K$ with $\zsh <\gs$.
\end{cor}

This solves problem 1.41 of \cite{Kir97}.

The constructive half of the proof of Theorem \ref{Thm:genusgap} relies on a reinterpretation of a classical technique of Lickorish \cite{Lic79} and Brakes \cite{Bra80} to produce $K$ and $K'$ with $X_0(K)$ diffeomorphic to $X_0(K')$. Lickorish and Brakes' technique can be used to produce all of the knots $K'$ that have been used to show $g_{sh}^{n}(K')<g_4(K')$ for $n\neq 0$ in the literature. However, all the knots $K, K'$ considered in any of the $n\neq 0$ literature have $g_4(K)\le 1$ and $g_4(K')\le 1$ and as such cannot be used to prove Theorem \ref{Thm:genusgap}.

With our new perspective (Theorem \ref{thm:diffeotraces}) we produce pairs of knots $K, K'$ with diffeomorphic 0-traces which a priori should be expected to have large  slice genus. In Theorem \ref{thm:tracesconc} we then give a more restrictive construction allowing us to build such a pair where $K$, surprisingly, has some prescribed small slice genus. This second construction is non-symmetric in $K$ and $K'$, so one still expects that $g_4(K')$ is large. We give an infinite family where this occurs, as well as several isolated examples. To give lower bounds on $g_4(K')$ we use Rasmussen's $s$ invariant. Hence 

\begin{cor}
Rasmussen's $s$ invariant is not a 0-trace invariant
\end{cor}

This addresses question 12 of \cite{4MKC16}. It is still unknown whether Ozsv{\'a}th-Szab{\'o}'s $\tau$ invariant is an invariant of the 0-trace of $K$.

Let $\sim$ denote smooth concordance and $\sC:= \{$ knots in $S^3\}/\sim$ be the concordance group. Any pattern $P$ in a solid torus a induces a well-defined map $P:\sC\to\sC$ by taking  $P([K]):=[P(K)]$ where $P$ acts on knots by the satellite operation. Since one is interested in understanding $\sC$ one might hope that $P$ sometimes induces an automorphism; in fact it unknown whether a satellite operator $P$ can ever induce a homomorphism of $\sC$ and it is conjectured that it cannot \cite{SST4M16}. More generally, one asks what sort of maps on $\sC$ can be obtained via the satellite operation. With an eye toward better understanding $\sC$, this is in particular an interesting question when attention is restricted to satellite operators with winding number 1.

It has been shown that winding number 1 satellite operators can induce non-surjective maps \cite{Lev16a} and interesting bijective maps \cite{GM95}\cite{MP17}, and that there exist patterns $J$ such that $J(U)\sim U$ but with $g_4(J(K))>g_4(K)$ for some $K$ \cite{CFHH13}. It is still unknown whether winding number 1 satellite operators can induce non-injective maps \cite{CR16}. Problem 7 of \cite{4MKC16} asked whether there exist winding number one patterns $J$ such that  $J(U)\sim U$ but with $g_4(J(K))<g_4(K)$ for some $K$; we resolve this as a corollary of our main theorem.

\begin{cor}\label{cor:satop}
There exist satellite operators $J$ such that $J(U)\sim U$ and such that $J$ induces a bijection on $\sC$ but such that $J$ does not preserve slice genus. 
\end{cor}
Satellite operators which induce bijections that fix the identity give candidates for automorphisms of $\sC$. Motivated by checking whether the examples used to prove Corollary \ref{cor:satop} are automorphisms, we give a pair of obstructions to a satellite map inducing a homomorphism on $\sC$. One of the obstructions is given below, the second is somewhat technical to state and is relegated to Section \ref{sec:maps}. While these obstructions are not particularly difficult and may be known to experts, we cannot find proofs in the literature so we produce them here. 

\begin{thm}\label{thm:nothom1}
Let $P$ be a winding number $w$ pattern and suppose that there exists some knot $K$ and additive slice genus bound $\sigma$ such that $$w g_4(K)<\sigma(P(K)).$$ Then $P$ is not a homomorphism.  
\end{thm}

\begin{cor}\label{prop:notauts}
The satellite operators used to prove Corollary \ref{cor:satop} do not induce automorphisms on $\sC$. 
\end{cor}

In a similar spirit, we show 
\begin{thm}\label{thm:stablegenus0}
Suppose a pattern $P$ induces a homomorphism on $\sC$. If $P$ has winding number 0 then $P(K)$ has stable slice genus 0 for all $K$. If $P$ has winding number one then $P(K)\#-K$ has stable slice genus 0 for all $K$. 
\end{thm}

The only known examples of knots with stable slice genus $0$ are the amphichiral knots, which have 2-torsion in $\sC$. As such, Theorem \ref{thm:stablegenus0} can be read as evidence that winding number 0 and 1 satellite operators never induce interesting homomorphisms. See Section \ref{sec:maps} for further discussion. 

All manifolds, submanifolds, maps of manifolds and concordances are smooth throughout this work, all homology has integer coefficients and all knots and manifolds are taken to be oriented. We will use $\cong$ to denote diffeomorphic manifolds, $\simeq$ to denote isotopic links, and $\sim$ to denote concordant knots. We will assume familiarity with handle calculus, for the details see \cite{GS99}.

\subsection*{Acknowledgements} The author previously explored many of these tools and objects in joint work with Allison N. Miller \cite{MP17} and that collaboration, as well as many subsequent insightful conversations, informs this work. Additionally, the author is grateful to Allison for pointing out an error in an early proof of Theorem \ref{thm:nothom1} and for comments on a draft of this paper. Conversations with Jeff Meier were clarifying and motivating and we are also grateful to Tye Lidman for helpful conversations.
We would like to acknowledge the developer of the Kirby Calculator \cite{Swe11} which we found helpful and sanity preserving while developing the examples in this work. Finally, the author is hugely grateful to her adviser, John Luecke, for extraordinary generosity with his encouragement, expertise, and time. \\
The author was partially supported by an NSF GRFP fellowship.

\section{Constructing knots with diffeomorphic traces}\label{sec:traces}

We begin by constructing pairs of knots with diffeomorphic traces.  Theorem \ref{thm:diffeotraces} was motivated by an inside-out take on the well-known dualizable patterns construction. 

Let $L$ be a three component link with (blue, green, and red) components $B,G,$ and $R$ such that the following hold: the sublink $B\cup R$ is isotopic in $S^3$ to the link $B\cup \mu_B$ where $\mu_B$ denotes a meridian of $B$, the sublink $G\cup R$ is isotopic to the link $G\cup \mu_G$, and lk$(B,G)=0$. From $L$ we can define an associated four manifold $X$ by thinking of $R$ as a 1-handle, in dotted circle notation, and $B$ and $G$ as attaching spheres of 0-framed 2-handles. See Figure \ref{fig:LforK} for an example of such a handle description. In a moment we will also define a pair of knots $K$ and $K'$ associated to $L$.
\begin{thm}\label{thm:diffeotraces}
$X\cong X(K)\cong X(K')$. 
\end{thm}

\begin{proof}
Isotope $L$ to a diagram in which $R$ has no self crossings (hence such that $R$ bounds a disk $D_R$ in the diagram) and in which $B\cap D_R$ is a single point. Slide $G$ over $B$ as needed to remove the intersections of $G$ with $D_R$. After the slides we can cancel the two handle with attaching circle $B$ with the one handle and we are left with a handle description for a 0-framed knot trace; this knot is $K'$. 

To construct $K$ and see $X\cong X(K)$, perform the above again with the roles of $B$ and $G$ reversed. 
\end{proof}




\begin{rmk}
By modifying the framing hypotheses in Theorem \ref{thm:diffeotraces} this technique can be easily modified to produce knots $J$ and $J'$ with $X_n(J)\cong X_n(J')$ for any integer $n$. 
\end{rmk}

For a link $L$ in $S^3$, define $-L$ to be the mirror of $L$ with its orientation reversed. Two $n$-component links $L_0$ and $L_1$ are said to be \emph{strongly concordant} if they cobound a smoothly embedded surface $\Sigma$ in $S^3\times [0,1]$ such that $\Sigma$ is a disjoint union of $n$ annuli and $\Sigma\cap (S^3\times\{0\})=-L_0$ and $\Sigma\cap (S^3\times\{1\})=L_1$. When $n=1$ we omit the word strongly.

\begin{figure}
\includegraphics[width=13cm]{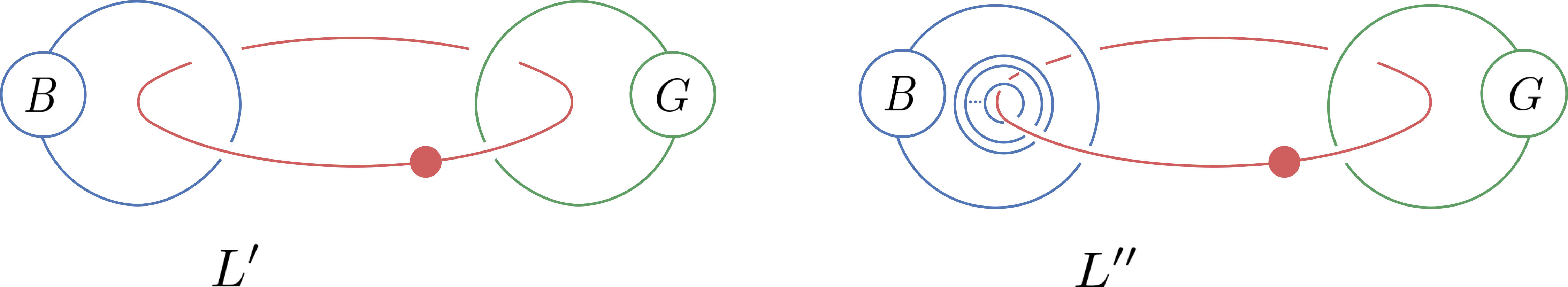}
\caption{}
\label{fig:lprimes}
\end{figure}
\begin{figure}
\includegraphics[width=5cm]{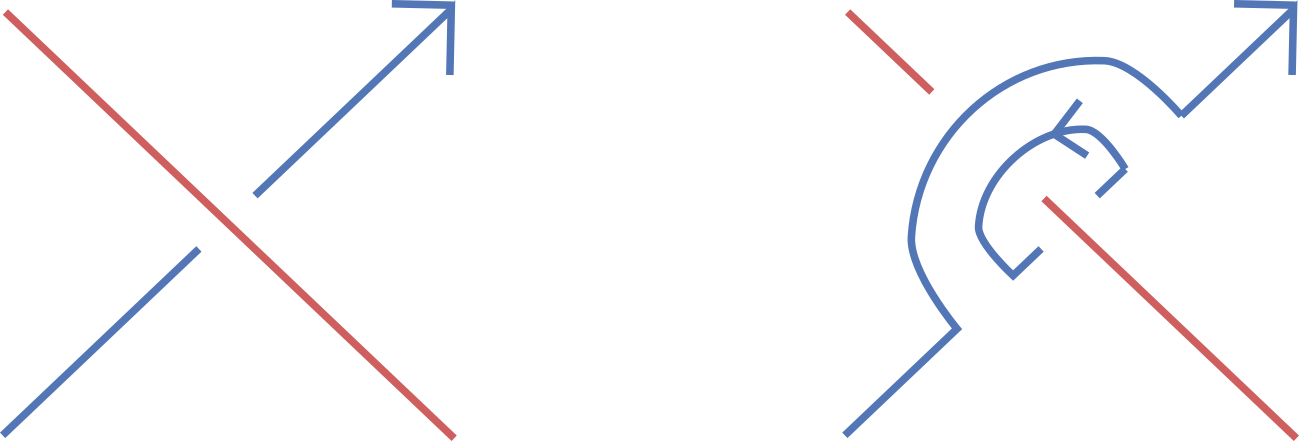}
\caption{}
\label{fig:bandcrossing}
\end{figure}
\begin{figure}
\includegraphics[width=5cm]{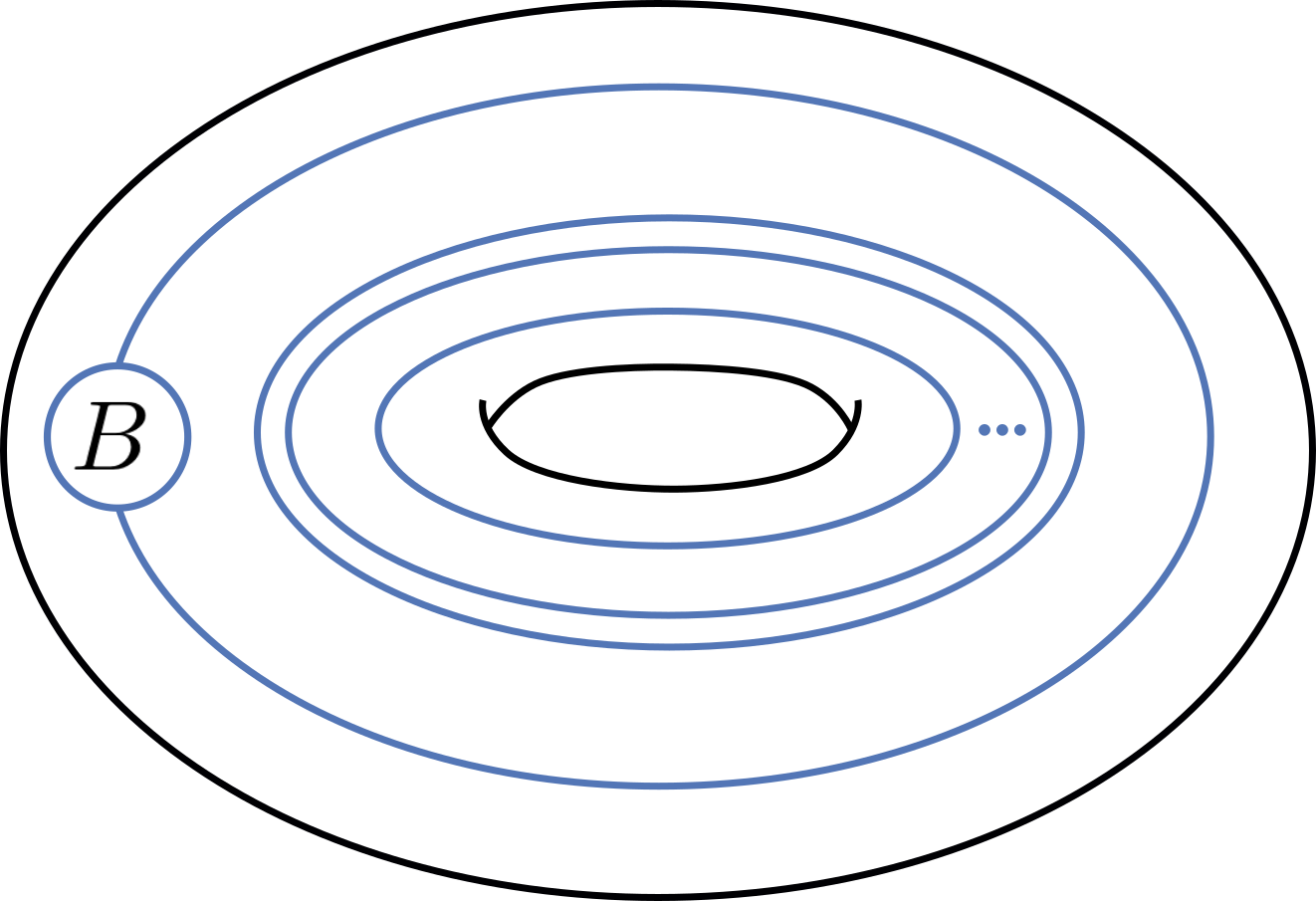}
\caption{}
\label{fig:Jpattern}
\end{figure}
\begin{thm}\label{thm:tracesconc}
Let $X$ be a four manifold with a handle description $L:=R\cup B\cup G$ as in Theorem \ref{thm:diffeotraces}. Further, suppose that $G\sim U$ and the link $B\cup G$ is split. If $K$ arises from $L$ as in Theorem \ref{thm:diffeotraces} then $K\sim B$.  
\end{thm}
\begin{proof}
Since $B\cup G$ is split, there exist finitely many crossing changes of $B$ with $R$ which change $L$ into the link $L'$ in Figure \ref{fig:lprimes}. As such there is a finite sequence of bands $\{\beta_i\}$ of $B$, as in Figure \ref{fig:bandcrossing}, which change $L$ into the link $L''$ in Figure \ref{fig:lprimes}. Then there is a diagram of $L$ obtained by performing the dual bands to given diagram of $L''$; isotope $L$ to this diagram and decorate it with the bands $\{\beta_i\}$. Then slide $B$ across $G$ at all points of $B\cap D_R$ so that we can cancel the 2-handle $G$ with the one handle. We obtain a diagram of $K\subset S^3$ decorated with bands, and such that when these bandings of $K$ are performed we obtain the link $J(G)$, where $J$ is the pattern in Figure \ref{fig:Jpattern}. Since $G$ is slice we obtain that the link $J(G)$ is strongly concordant to $J(U)$, hence $K$ is concordant to $B$. 
\end{proof}

\begin{figure}

\includegraphics[width=12cm]{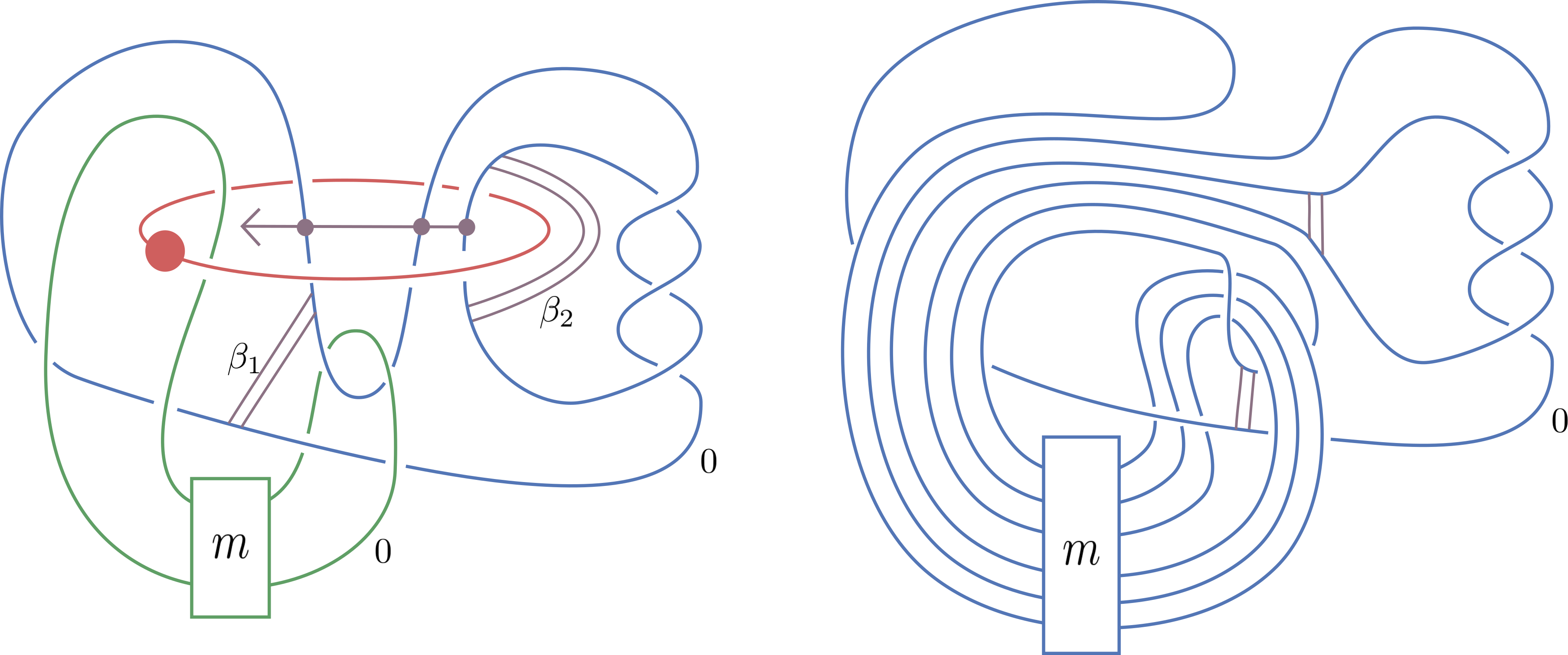}
\caption{A handle diagram for $X_m$ and diffeomorphism to $X_0(K_m)$}
\label{fig:LforK}
\end{figure}

\begin{ex}\label{ex:Km}
Let $m$ be an integer and $L_m$ be the decorated link on the left hand side of Figure \ref{fig:LforK}, which describes a four manifold $X_m$, and observe that $L_m$ satisfies the hypotheses of Theorem \ref{thm:tracesconc}. After the indicated slides we obtain a diagram of $X_m$ as the 0-trace of a knot we call $K_m$. By Theorem \ref{thm:tracesconc}, $K_m$ is concordant to $B$ which we see is isotopic to the right-hand trefoil for all $m$.  
\begin{figure}
\includegraphics[width=12cm]{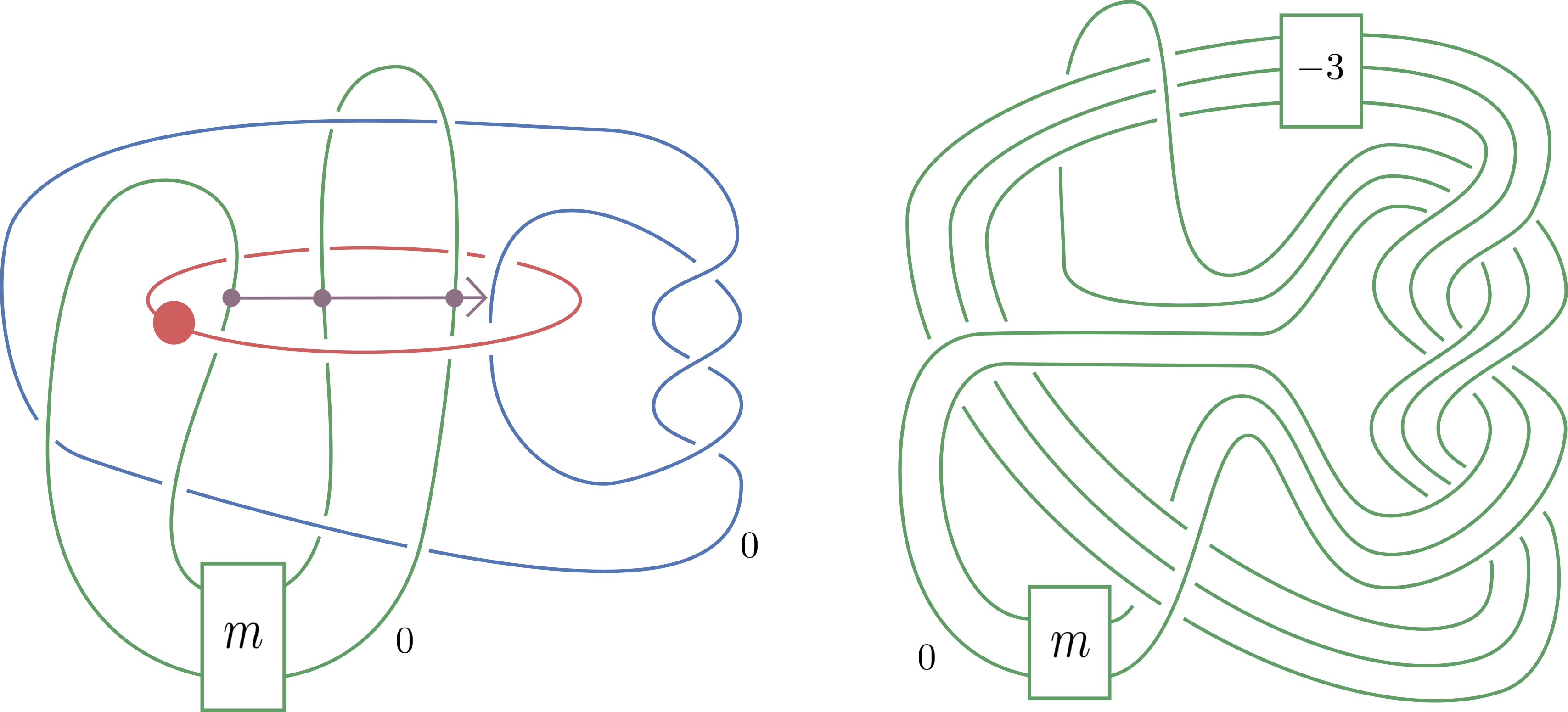}
\caption{The same handle diagram for $X_m$ and diffeomorphism to $X_0(K'_m)$}
\label{fig:LforKprime}
\end{figure}
We then isotope $L_m$ to get a handle diagram for $X_m$ as the 0-trace of a knot we call $K'_m$. See Figure \ref{fig:LforKprime}.
\end{ex}

\begin{rmk}
In Figure \ref{fig:LforK} we have illustrated bands $\{\beta_i\}$ such that banding along $\{\beta_i\}$ in the left hand diagram changes $B$ into a three component link split from $G$, where two components are isotopic to $\mu_R$, as in the proof of Theorem \ref{thm:tracesconc}. We also kept track of $\{\beta_i\}$ through the diffeomorphism. In practice neither exhibiting nor keeping track of the bands is necessary; we have included it here to build intuition for the proof of Theorem \ref{thm:tracesconc} and demonstrate how Theorem \ref{thm:tracesconc} can be used to give an explicit description of the implied concordance. 
\end{rmk}

The diagram we give of $K_m$ in Figure \ref{fig:LforK} can certainly be simplified, but since we will only be concerned with $K_m$ up to concordance and we understand $[K_m]$ by Theorem \ref{thm:tracesconc}, we don't pursue this. This illustrates the usefulness of Theorem \ref{thm:tracesconc}; if one wants to compare the concordance properties of knots with diffeomorphic traces one can get a tractable pair by choosing $L$ so that $K'$ remains relatively simple (in crossing number perhaps, or whatever is convenient) and since we understand $[K]$ it does not matter if the knot $K$ is complicated. 

\begin{figure}

\includegraphics[width=5cm]{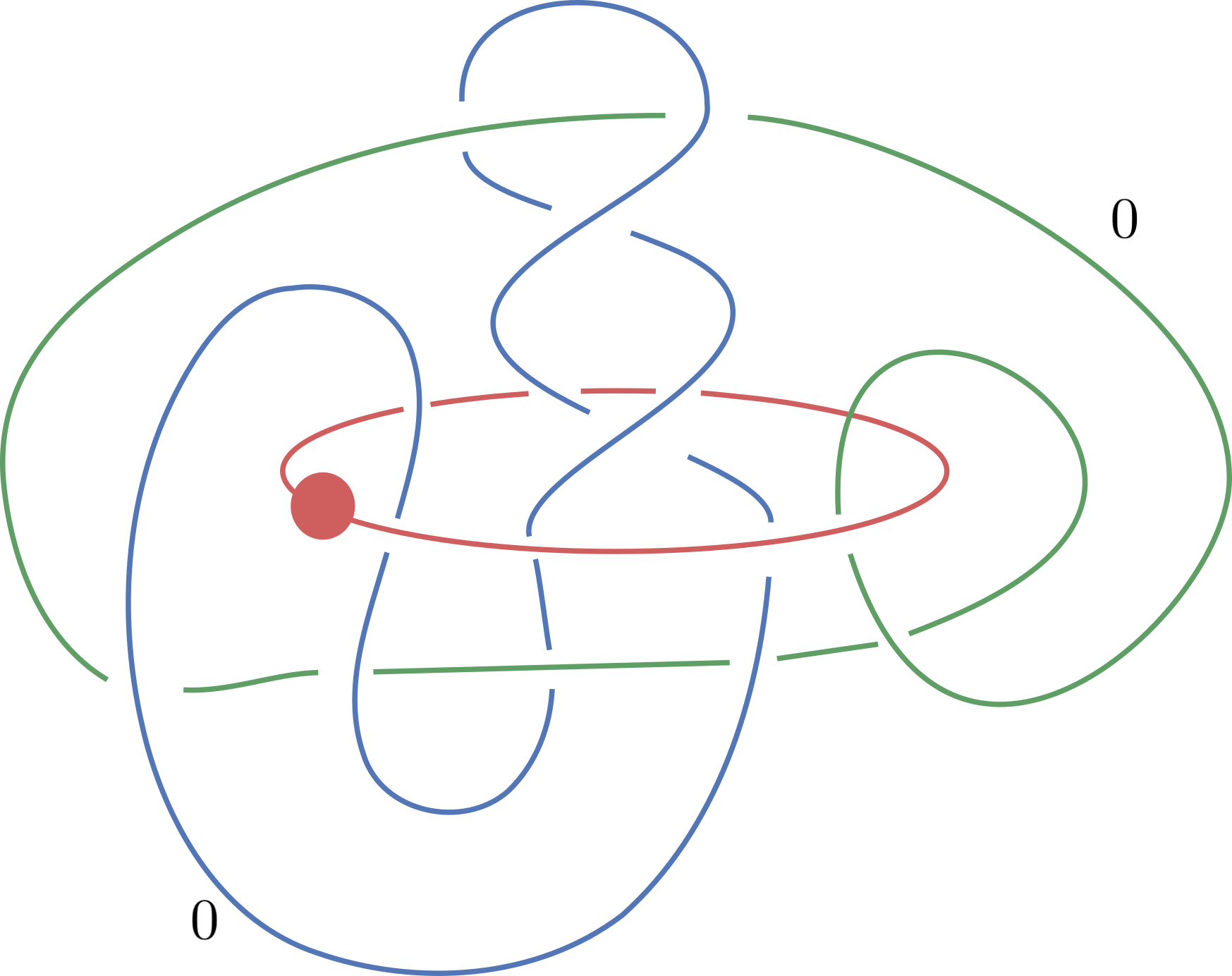}
\caption{}
\label{fig:splithyp}
\end{figure}

\begin{rmk}
The split hypothesis of Theorem \ref{thm:tracesconc} is essential. For example, consider the handle diagram $L$ in Figure \ref{fig:splithyp} and let $K$ be the knot obtained from $L$ as in Theorem \ref{thm:diffeotraces}. $K$ is isotopic to the pretzel knot $P(5,-3,-3)$, which is not slice. 
\end{rmk}

\section{Rasmussen's invariant calculations}

In \cite{Kho00}, Khovanov introduced a link invariant which takes the form of a bigraded abelian group. We refer to this group as the Khovanov homology $Kh^{i,j}(L)$, and will sometimes find it convenient to refer to the Poincare polynomial of this group, which is $$Kh(L)(q,t):=\sum_{i,j}q^jt^i \text{rank}(Kh^{i,j}(L))$$ Later, Lee showed that $Kh(L)$ can be viewed as the $E_2$ page of a spectral sequence which converges to $\bQ\oplus\bQ$ \cite{Lee02} and Rasmussen used this to define an integer valued knot invariant $s(K)$ \cite{Ras10}. It will suffice for this work to recall the following properties of $s(K)$. 
\begin{thm}[\cite{Ras10}]\label{thm:sproperties} For any knot $K$ in $S^3$, the following hold:
\begin{enumerate}
\item \hspace{10pt}
$|s(K)|\le 2g_4(K)$
\item \hspace{10pt} The map $s$ induces a homomorphism from $\sC$ to $\bZ$.
\item \hspace{10pt} $\text{rank}(Kh(K)^{0,s(K)\pm 1})\neq 0$
\end{enumerate}
\end{thm}
\begin{cor}[\cite{Ras10}]\label{cor:scrossingchange}
Suppose $K_+$ and $K_-$ are knots that differ by a single crossing change, from a positive crossing in $K_+$ to a negative one in $K_-$. Then $s(K_-)\le s(K_+)\le s(K_-)+2$. 
\end{cor}

\begin{figure}[]
\includegraphics[width=6cm]{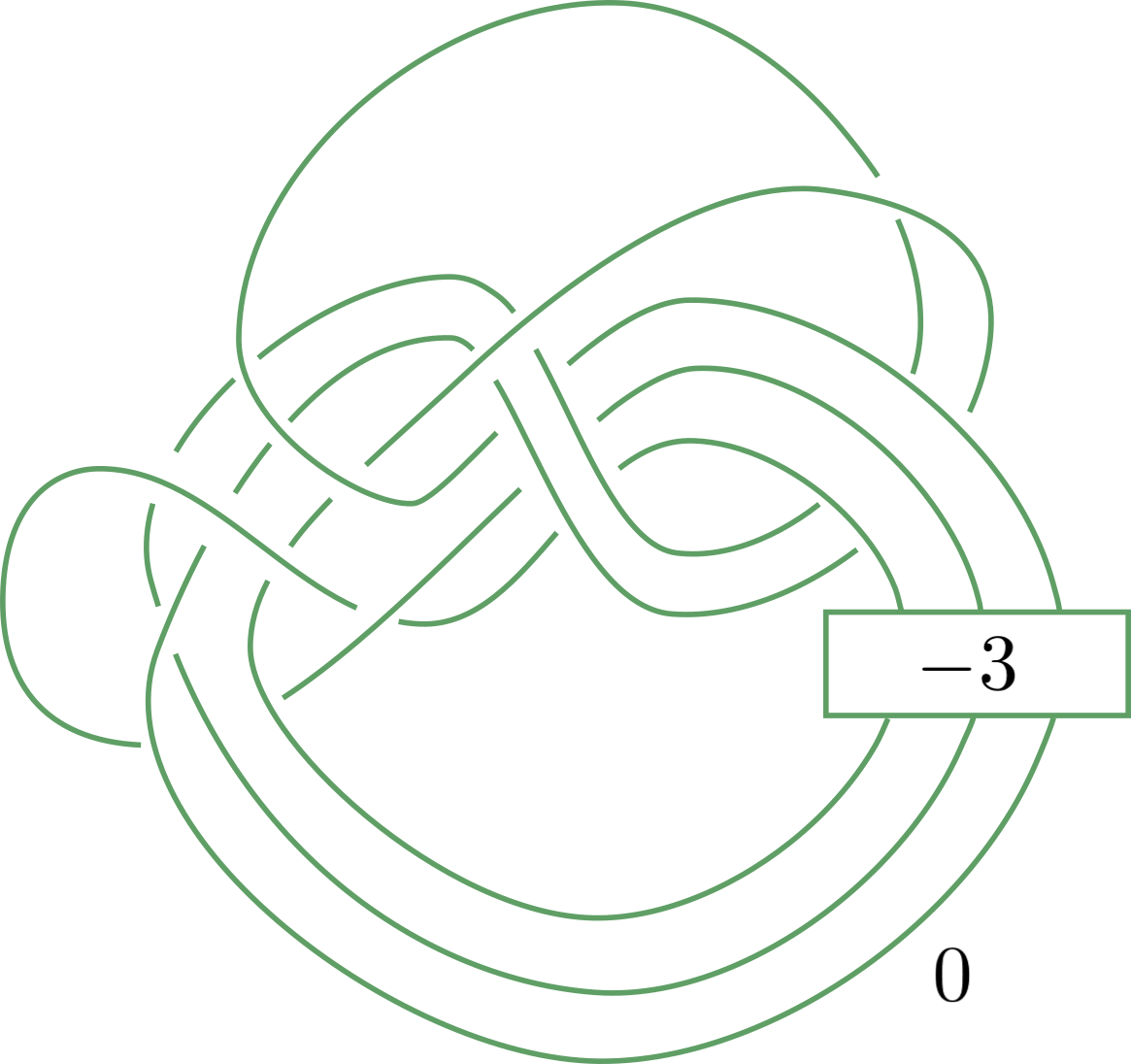}
\caption{}
\label{fig:Kmprimereduced}
\end{figure}

\begin{proof}[Proof of Theorem \ref{Thm:genusgap}]
For a fixed $m\le 0$ let $K_m$ and $K'_m$ be the knots from Example \ref{ex:Km}. By Theorem \ref{thm:diffeotraces} $X_0(K_m)\cong X_0(K'_m)$, and as remarked in Example \ref{ex:Km}, $g_4(K_m)=1$ for all $m$. 
For $m\le 0$ we will bound the slice genus of $K'_m$ from below by bounding $s(K'_m)$ from below. See Figure \ref{fig:Kmprimereduced} for a somewhat reduced diagram of $K'_0$ with approximately 40 crossings. We make use of the JavaKh routines, available at \cite{KAT} to compute $Kh^{i,j}(K'_0)$. These routines produce Poincare polynomial
\begin{equation}
\begin{split}
Kh(K'_0)(q,t)=&q^{47} t^{27}+q^{43} t^{26}+q^{41} t^{24}+q^{41} t^{23}+q^{39} t^{22}+q^{37} t^{23}+q^{37} t^{22}+q^{37} t^{20}+q^{35} t^{21}\\+
&q^{35} t^{20}+2 q^{33} t^{19}+q^{33} t^{18}+q^{31} t^{19}+q^{31} t^{17}+2 q^{31} t^{16}+q^{29} t^{18}+q^{29} t^{17}+q^{29} t^{15}\\+
&q^{27} t^{16}+3 q^{27} t^{15}+q^{27} t^{14}+q^{25} t^{14}+2 q^{25} t^{13}+q^{25} t^{12}+q^{23} t^{14}+q^{23} t^{13}+q^{23} t^{12}\\
&+q^{23} t^{11}+2 q^{21} t^{12}+2 q^{21} t^{11}+q^{19} t^{11}+q^{19} t^{10}+q^{19} t^9+q^{19} t^8+q^{17} t^{10}+q^{17} t^8\\
&+q^{17} t^7+q^{15} t^8+q^{15} t^7+q^{13} t^7+q^{13} t^6+q^{13} t^5+q^{13} t^4+q^{11} t^4+q^{11} t^3\\
&+q^9 t^4+2 q^9 t^3+q^9 t^2+q^7 t^3+q^7 t^2+q^7 t+q^5 t^2+q^5 t+q^5+\frac{q^3}{t}+2 q^3+\frac{1}{q t^2}\\
\end{split}
\end{equation}

Then, by item 3 of Theorem  \ref{thm:sproperties} we have $s(K'_0)=4$, and by Corollary \ref{cor:scrossingchange} we have $s(K'_m)\ge 4$ for all $m\le 0$. We conclude by appealing to item 1 of  Theorem \ref{thm:sproperties}. 
\end{proof}
\begin{rmk}
It is not hard to check that $g_4(K'_m)\le 2$ for all $m\in\bZ$. Hence all bounds in the above proof are sharp. 
\end{rmk}
We now produce another isolated example of a pair of knots $K, K'$ with $X_0(K)\cong X_0(K')$ and $g_4(K)\neq g_4(K')$. This example could also be expanded to infinite families as done with Example \ref{ex:Km}, but we do not pursue that here. This and the the knots in Example \ref{ex:Km} are only special in that they have $K'$ with reasonably small crossing number, allowing us to compute $Kh(K')$. We anticipate that Theorem \ref{thm:tracesconc} can be used to give abundant examples of knots $K, K'$ with diffeomorphic traces and distinct slice genera. 

\begin{ex}
\begin{figure}[]
\includegraphics[width=13cm]{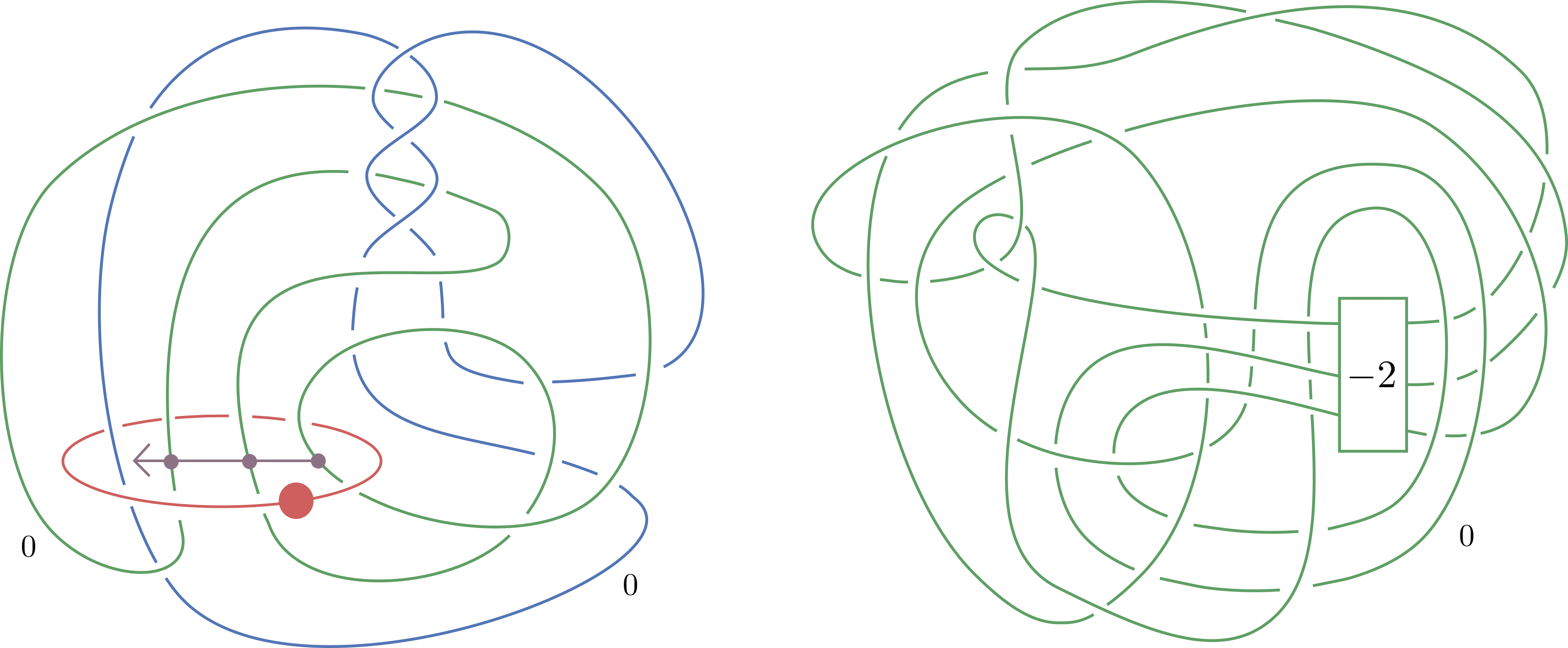}
\caption{}
\label{fig:Ex3}
\end{figure}
Let $L$ be the decorated link in the left hand side of Figure \ref{fig:Ex3}. By Theorems \ref{thm:diffeotraces} and \ref{thm:tracesconc}, $L$ gives a handle decomposition for $X\cong X_0(K)$ where $K$ is concordant to the right hand trefoil. By performing the slides indicated in the left hand side of Figure \ref{fig:Ex3} one obtains a knot $K'$ with $X\cong X_0(K')$, not pictured.  In the right hand side of Figure \ref{fig:Ex3} we give a diagram of a knot $K''$ with fewer crossings; we claim $K''$ is isotopic to $K'$. The isotopy between the diagram of $K'$ described and the diagram of $K''$ given is non-trivial, we provide two options for the careful reader to confirm the existence of an isotopy. First, they can use Snappea \cite{Wee01} to confirm that the diagrams present isotopic knots. Alternatively they can use the Kirby calculator \cite{Swe11} to view the diffeomorphism $f: X_0(K')\to X_0(K'')$ which we have located on the author's website \cite{Pic18} and can check that this diffeomorphism sends a 0-framed copy of $\mu_{K'}$ to $\mu_{K''}$. We warn that the diffeomorphism $f$ is tedious. 

Having confirmed that $K'\simeq K''$ it suffices to compute $s(K'')$. As before we rely on the JavaKh routines and item 3 of \ref{thm:sproperties}. We suppress the output of JavaKh and present the conclusion, which is that $s(K'')\in\{4,6\}$. Since one checks that $g_4(K')\le 2$ we conclude $s(K')=4$ and $g_4(K')=2$. 
\end{ex}

\section{Bijective maps on $\sC$ which do not preserve slice genus, and some satellite homomorphism obstructions}\label{sec:maps}

\subsection{Definitions and notation}

Let $P:\sqcup_n S^1\to V$ be an oriented link in a parametrized solid torus $V:=S^1 \times D^2$.  By the usual abuse of notation, we use $P$ to refer to both this map and its image.  Define $\lambda_V=S^1\times \{x_0\}$ for some $x_0\in\partial D^2$, oriented so that $P$ is homologous to a non-negative multiple $n$ of $\lambda_V$. We call $n$ the \emph{(algebraic) winding number of} $P$. Define the \emph{geometric winding number} of $P$ to be the minimal number of intersections of $P$ with the meridional disk for $V$ over all patterns in the isotopy class of $P$. 

Given a pattern $P: S^1 \to V$, define $\overline{P}$ to be the pattern obtained from $P$ by reversing the orientation of both $S^1$ and $V$; note that $\overline{P}$ has diagram obtained from a diagram of $P$ by changing all crossings and the orientation of $P$. 

For any knot $K$ in $S^3$ there is a canonical embedding $i_K:V \to S^3$ given by identifying $V$ with $\overline{\nu(K)}$ such that $\lambda_V$ is sent to the null-homologous curve on $\partial\overline{\nu(K)}$. Then $i_K \circ P:S^1\to S^3$ specifies an oriented knot in $S^3$, denoted $P(K)$ and called the \emph{satellite of $K$ by $P$}. As such the pattern $P$ gives a map from $\{$knots in $S^3\}\to \{$knots in $S^3\}$, which we will also refer to as $P$. It is not hard to show that $P$ descends to a well-defined map $P:\sC\to\sC$.

\subsection{Bijective operators not preserving slice genus}

\begin{thm}[Proposition 6.13 of \cite{CR16}]\label{thm:CRiff}
For a knot $K$ in $S^3$, $g_{sh}^0(K)=g_4(K)$ if and only if $g_4(P(K))\ge g_4(K)$ for all winding number one satellite operators $P$ with $P(U)$ slice.
\end{thm}

If one ignores the `bijective' conclusion, then Corollary \ref{cor:satop} follows immediately from Corollary \ref{cor:shakenotslice} and Theorem \ref{thm:CRiff}. It is also possible, though quite tedious, to use the techniques of Cochran-Ray's proof of Theorem \ref{thm:CRiff} to construct an explicit pattern $Q_m$, not necessarily bijective, which lowers the slice genus of the knots $K_m'$ from the proof of Theorem \ref{Thm:genusgap}. Instead, we show in this section that dualizable patterns, a classical technique for constructing knots with diffeomorphic 0-traces, readily yield such a $Q_m$ which is bijective. 

In order to prove Corollary \ref{cor:satop} and give the examples, we state the facts we require about dualizable patterns now. The proof of Proposition \ref{prop:dualequiv} requires recalling the dualizable patterns construction and is thus located in the Appendix. It will suffice for this work to recall the following properties of dualizable patterns; we give the definition in the Appendix. 

\begin{prop}\label{prop:dualequiv}For any knots $K, K'$ that arise from Theorem \ref{thm:diffeotraces} there exists a dualizable pattern $P$ with dual pattern $P^*$ such that $P(U)\simeq K$ and $P^*(U)\simeq K'$. Conversely for any dualizable pattern $P$ one can express $P(U)\simeq K$ and $P^*(U)\simeq K'$ with $K$ and $K'$ arising as in Theorem \ref{thm:diffeotraces}.
\end{prop}

The proof of the above is constructive; in particular given $K, K'$ the proof illustrates how to write down an associated dualizable pattern $P$. 

\begin{thm}[Proposition 2.4 of \cite{GM95}, Theorem 1.12 of \cite{MP17}]\label{thm:bijective}
For any dualizable pattern $P$ and knot $K$ in $S^3$, we have
$$P(\overline{P^*}(K))\sim K\sim \overline{P^*}(P(K))$$ In other words, dualizable patterns induce bijective maps on $\sC$. 
\end{thm}
\begin{prop}[Proposition 4.3 of \cite{MP17}]\label{prop:dualcomp}
If patterns $J, P$ are both dualizable then so is $P\circ J$. 
\end{prop}
\begin{defin}
For a pattern $P$, define $P_{\#}$ to be the geometric winding number one pattern with $P_{\#}(U)\simeq P(U)$. 
\end{defin}
\begin{rmk}\label{rmk:lightbulbdual}
All geometric winding number one patterns are dualizable. 
\end{rmk}
\begin{figure}[]
\includegraphics[width=6cm]{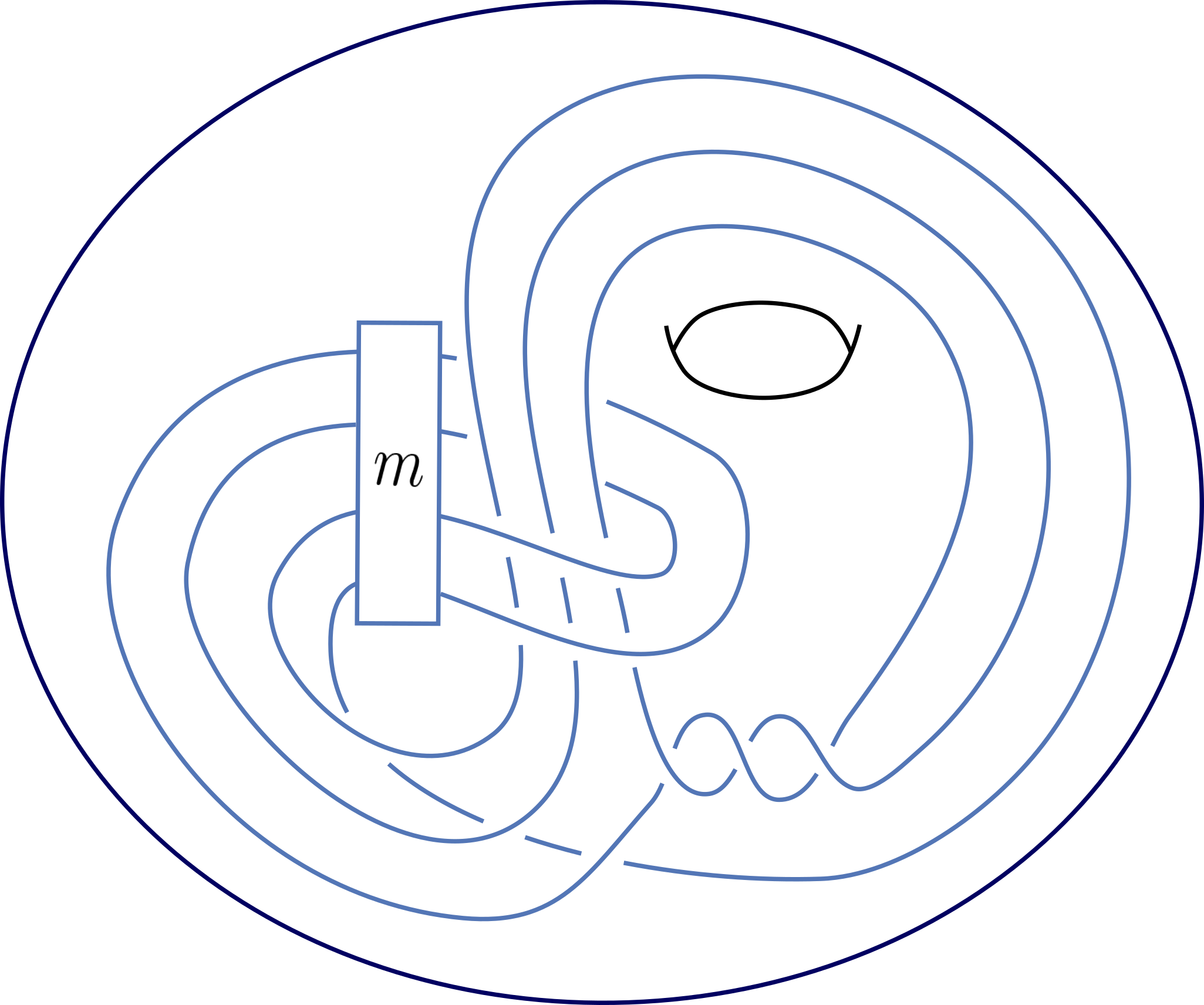}
\caption{}
\label{fig:Pm}
\end{figure}
\begin{proof}[Proof of Corollary \ref{cor:satop}]
Let $K_m$ and $K_m'$ be the knots from the proof of Theorem \ref{Thm:genusgap} and let $P_m, P_m'$ be the dualizable patterns with $P_m(U)\simeq K_m$ and $P_m'(U)\simeq K_m'$ as in Proposition \ref{prop:dualequiv}. The pattern $P_m$ is illustrated in Figure \ref{fig:Pm}. Define the pattern $Q_m:=(P_{m})_{\#}\circ\overline{P_m}$. By Proposition \ref{prop:dualcomp} and Remark \ref{rmk:lightbulbdual} $Q_m$ is dualizable, hence by Theorem \ref{thm:bijective} is bijective on $\sC$, and we see that $Q_m(U)\simeq P_m(U)\#\overline{P_m}(U)$ is slice. We conclude by observing $$Q_m(K_m')\simeq Q_m(P_m'(U))\simeq P_m(U)\#\overline{P_m}(P_m'(U))\sim P_m(U)\simeq K_m$$ where the concordance follows from Theorem \ref{thm:bijective}. 

\end{proof}

\subsection{Satellite homomorphism obstructions}

Satellite operators which induce bijections that fix the identity are a priori candidates for automorphisms of $\sC$. Motivated by studying this for our examples, we prove Theorems \ref{thm:nothom1} and \ref{thm:nothom2}.  

We use the shorthand $nK$ to denote $\#_n(K)$ and define an additive slice genus bound to be any knot invariant $\sigma$ with $\sigma(K)\le g_4(K)$ and $\sigma(J\#K)\ge \sigma(J)+\sigma(K)$ for all knots $K$ and $J$. The classical knot signature, Tristam-Levine signatures, J. Rasmussen's $s$ invariant, Oszvath-Szabo's $\tau$ invariant and many other concordance invariants from the $HFK^{\infty}$ package all give examples of additive slice genus bounds. The slice genus is \emph{not} an additive slice genus bound.

We will require 
\begin{prop}[Proposition 6.3 of \cite{CH14}]\label{prop:bdddistance}
For any winding number $w$ satellite operators $P$ and $J$ there is a constant $C_{P,J}$ such that for all knots $K$ $$|g_4(J(K))-g_4(P(K))|\le C_{P,J}$$
\end{prop}
The proposition follows from observing that in $S^1\times D^2\times [0,1]$ there exists some genus $g$ cobordism between $P$ and $J$, and taking $C_{P,J}:=g$. For details, see \cite{CH14}. 

\begin{proof}[Proof of Theorem \ref{thm:nothom1}]
Suppose $P$ does induce a homomorphism. Then for all $n\in\bZ^+$ we have $$g_4(P(nK))=g_4(nP(K))\ge \sigma(nP(K))\ge n\sigma(P(K))\ge n w g_4(K)+nr$$ for some positive $r$ which is independent from $n$. 

By Proposition \ref{prop:bdddistance} there exists a constant $C_{P,J}$ such that $g_4(P(nK))\le C_{P,J}+g_4(J(nK))$ where $J$ is the $(w,1)$ cable operator. But we also have
$$C_{P,J}+g_4(J(nK))\le C_{P,J}+wg_4(nK)\le C_{P,J}+wng_4(K) $$
Hence $$nr+|w|ng_4(K)\le C_{P,J}+|w|ng_4(K)$$
By taking $n$ large we get a contradiction.
\end{proof}

\begin{thm}\label{thm:nothom2}
Suppose $P$ has winding number $w$ and that there exists a knot $K$, a winding number $w$ pattern $J$, an additive slice genus bound $\sigma$, and a slice genus bound $\overline{\sigma}$ such that
\begin{enumerate}
\item $\sigma (P(K))> g_4(K)$
\item $\overline{\sigma}(J(K_0))\ge \sigma(K_0)$ for all knots $K_0$. 
\end{enumerate}
Then $P$ is not a homomorphism. 
\end{thm}

\begin{rmk}
(2) is a technical condition. When $w=1$, we always have (2) by taking $\overline{\sigma}=\sigma$ and $J$ to be the identity pattern. When $w=0$ we never have (2), which is consistent since the trivial pattern is a winding number 0 homomorphism which  lowers slice genus of infinitely many $K$. If $|w|>1$ and $\sigma$ behaves sufficiently well on cables then we have (2) by taking $J$ an appropriate cable and $\overline{\sigma}=\sigma$. For example when $\sigma\in\{\tau, s\}$ taking $J$ to be the $(w,1)$ cable works. 
\end{rmk}
\begin{proof}
Suppose $P$ is a homomorphism. By hypothesis we have $$\sigma(nK)\ge n\sigma(K)\ge nr+ng_4(P(K))$$ for some $r>0$ and all $n\ge 0$. By Proposition \ref{prop:bdddistance} there exists a constant $C_{P,J}$ such that $$g_4(J(nK))\le C_{P,J}+g_4(P(nK))=C_{P,J}+g_4(nP(K))\le C_{P,J}+ng_4(P(K))$$ where $J$ is the $(w,1)$ cable operator. Since by hypothesis $$g_4(J(nK))\ge\overline{\sigma}(J(nK))\ge\sigma(nK),$$ we get $$C_{P,J}+ng_4(P(K))\ge nr+ng_4(P(K)).$$ By taking $n$ large, we get a contradiction.
\end{proof}

Taken together, Theorems \ref{thm:nothom1} and \ref{thm:nothom2} indicate loosely that any winding number $w\neq 0$ satellite homomorphism $P$ must have $g_4(K)\le g_4(P(K))\le |w|g_4(K)$ for all knots $K$, at least insofar as can be detected by any additive slice genus bound. 

We conclude by proving Theorem \ref{thm:stablegenus0} as a final application of these ideas.
\begin{defin}[See \cite{Liv10}]
The stable four genus of a knot $K$ is defined to be $$g_{st}(K):=\lim_{n\to\infty}\frac{g_4(nK)}{n}$$
\end{defin}
It is not known whether there exist any non-amphichiral knots $K$ with $g_{st}(K)=0$ or whether $g_{st}(K)=0$ implies that $K$ is torsion on $\sC$ \cite{Liv10}. It is interesting then that the existence of certain satellite homomorphisms gives rise to abundant examples of knots with stable four genus 0 as follows.
\begin{proof}[Proof of Theorem \ref{thm:stablegenus0}]
Suppose $P$ and $J$ are any winding number $w$ satellite homomorphisms. By Proposition \ref{prop:bdddistance} there is some constant $C_{P,J}$ with $g_4(P(nK)-J(nK))\le C_{P,J}$ for all knots $K$ and integers $n$. Since $P$ and $J$ are homomorphisms $P(nK)-J(nK)\sim nP(K)-nJ(K)\simeq n(P(K)-J(K))$. Hence $g_4(n(P(K)-J(K)))\le C_{P,J}$, so $P(K)-J(K)$ has stable genus 0. The conclusions follow by observing that the identity and zero maps arise as winding number 1 and 0 satellite homomorphisms, respectively.
\end{proof}

\bibliography{satellite}

\begin{thebibliography}{CFHH13}

\bibitem[4MK16]{4MKC16}
Problem list.
\newblock In {\em Conference on 4-manifolds and knot concordance}, Max Planck
  Institute for Mathematics, Oct.~17-21 2016.

\bibitem[AJOT13]{AJOT13}
Tetsuya Abe, In~Dae Jong, Yuka Omae, and Masanori Takeuchi.
\newblock Annulus twist and diffeomorphic 4-manifolds.
\newblock In {\em Mathematical Proceedings of the Cambridge Philosophical
  Society}, volume 155, pages 219--235. Cambridge University Press, 2013.

\bibitem[Akb77]{Akb77}
Selman Akbulut.
\newblock On {$2$}-dimensional homology classes of {$4$}-manifolds.
\newblock {\em Math. Proc. Cambridge Philos. Soc.}, 82(1):99--106, 1977.

\bibitem[Akb93]{Akb93}
S~Akbulut.
\newblock Knots and exotic smooth structures on 4-manifolds.
\newblock {\em Journal of Knot Theory and Its Ramifications}, 2(01):1--10,
  1993.

\bibitem[Bra80]{Bra80}
W.~R. Brakes.
\newblock Manifolds with multiple knot-surgery descriptions.
\newblock {\em Math. Proc. Cambridge Philos. Soc.}, 87(3):443--448, 1980.

\bibitem[CFHH13]{CFHH13}
Tim Cochran, Bridget Franklin, Matthew Hedden, and Peter Horn.
\newblock Knot concordance and homology cobordism.
\newblock {\em Proceedings of the American Mathematical Society},
  141(6):2193--2208, 2013.

\bibitem[CH14]{CH14}
Tim~D Cochran and Shelly Harvey.
\newblock The geometry of the knot concordance space.
\newblock {\em arXiv preprint arXiv:1404.5076}, 2014.

\bibitem[COT03]{COT03}
Tim~D. Cochran, Kent~E. Orr, and Peter Teichner.
\newblock Knot concordance, {W}hitney towers and {$L^2$}-signatures.
\newblock {\em Ann. of Math. (2)}, 157(2):433--519, 2003.

\bibitem[CR16]{CR16}
Tim~D. Cochran and Arunima Ray.
\newblock Shake slice and shake concordant knots.
\newblock {\em J. Topol.}, 9(3):861--888, 2016.

\bibitem[GM95]{GM95}
Robert~E. Gompf and Katura Miyazaki.
\newblock Some well-disguised ribbon knots.
\newblock {\em Topology Appl.}, 64(2):117--131, 1995.

\bibitem[GS99]{GS99}
Robert~E. Gompf and Andr\'as~I. Stipsicz.
\newblock {\em {$4$}-manifolds and {K}irby calculus}, volume~20 of {\em
  Graduate Studies in Mathematics}.
\newblock American Mathematical Society, Providence, RI, 1999.

\bibitem[KAT]{KAT}
The knot atlas.
\newblock {\em http://katlas.org/}.

\bibitem[Kho00]{Kho00}
Mikhail Khovanov.
\newblock A categorification of the {J}ones polynomial.
\newblock {\em Duke Math. J.}, 101(3):359--426, 2000.

\bibitem[Kir97]{Kir97}
Problems in low-dimensional topology.
\newblock In Rob Kirby, editor, {\em Geometric topology ({A}thens, {GA},
  1993)}, volume~2 of {\em AMS/IP Stud. Adv. Math.}, pages 35--473. Amer. Math.
  Soc., Providence, RI, 1997.

\bibitem[KM78]{KM78}
Robion Kirby and Paul Melvin.
\newblock Slice knots and property {${\rm R}$}.
\newblock {\em Invent. Math.}, 45(1):57--59, 1978.

\bibitem[Lee08]{Lee02}
Eun~Soo Lee.
\newblock On khovanov invariant for alternating links.
\newblock {\em arXiv preprint math.GT/0210213}, 2008.

\bibitem[Lev16]{Lev16a}
Adam~Simon Levine.
\newblock Nonsurjective satellite operators and piecewise-linear concordance.
\newblock In {\em Forum of Mathematics, Sigma}, volume~4. Cambridge University
  Press, 2016.

\bibitem[Lic79]{Lic79}
WB~Raymond Lickorish.
\newblock Shake slice knots.
\newblock In {\em Topology of Low-Dimensional Manifolds}, pages 67--70.
  Springer, 1979.

\bibitem[Liv10]{Liv10}
Charles Livingston.
\newblock The stable 4--genus of knots.
\newblock {\em Algebraic \& Geometric Topology}, 10(4):2191--2202, 2010.

\bibitem[MP18]{MP17}
Allison~N Miller and Lisa Piccirillo.
\newblock Knot traces and concordance.
\newblock {\em Journal of Topology}, 11(1):201--220, 2018.

\bibitem[Pic]{Pic18}
home.
\newblock \url{https://www.ma.utexas.edu/users/lpiccirillo/}.
\newblock (Accessed on 02/26/2018).

\bibitem[Ras10]{Ras10}
Jacob Rasmussen.
\newblock Khovanov homology and the slice genus.
\newblock {\em Inventiones mathematicae}, 182(2):419--447, 2010.

\bibitem[SST16]{SST4M16}
Problem list.
\newblock In {\em Synchronisation of smooth and topological 4-manifolds}, Banff
  International Research Station, Banff, Canda, Feb.~21-26 2016.

\bibitem[Swe11]{Swe11}
Frank Swenton.
\newblock Kirby calculator.
\newblock {\em URL: http://community. middlebury. edu/\~{}
  mathanimations/kirbycalculator/[cited April 16, 2015]}, 2011.

\bibitem[Wee01]{Wee01}
Jeff Weeks.
\newblock Snappea: A computer program for creating and studying hyperbolic
  3-manifolds, 2001.

\bibitem[Yas15]{Yas15a}
Kouichi Yasui.
\newblock Corks, exotic 4-manifolds and knot concordance.
\newblock {\em arXiv preprint arXiv:1505.02551}, 2015.

\end{thebibliography}

\section{Appendix}\label{sec:appen}
Herein we define dualizable patterns, which are the fundamental object used in the original construction of knots with diffeomorphic traces, and prove Proposition \ref{prop:dualequiv}. The definition of dualizable patterns was inspired by examples  of Akbulut \cite{Akb77} and was developed and formalized in work of Lickorish \cite{Lic79} and Brakes \cite{Bra80} independently at around the same time. Several recent papers on the subject of constructing knots with diffeomorphic traces, including one by the author, have erroneously failed to reference Lickorish's work. 

To define dualizable patterns, we fix some conventions. For a pattern $P\in S^1\times D^2=:V$ define $\mu_P$ to be a  meridian for $P$, oriented such that the linking number of $P$ and $\mu_P$ is 1, and define $\mu_V=\{x_1\}\times \partial D^2$ for some $x_1\in S^1$, oriented so that $\mu_V$ is homologous to a non-negative multiple of $\mu_P$. Finally define the longitude $\lambda_P$ of $P$ to be the unique framing curve of $P$ in $V$ which is homologous to a positive multiple of $\lambda_V$ in $V\smallsetminus \nu(P)$. 

\begin{defin}
Define $\Gamma:S^1\times D^2 \to S^1\times S^2$ by $\Gamma(t,d)=(t, \gamma(d))$, where   $\gamma:D^2\to S^2$ is an arbitrary orientation preserving embedding.
 Then for any curve $\alpha: S^1 \to S^1 \times D^2$, we can define a knot in $S^1\times S^2$ by
$\widehat{\alpha} :=\Gamma \circ \alpha:S^1\to S^1\times S^2.$
 \end{defin}

We warn the reader that Definition \ref{defn:dualizable} is somewhat non-standard; for the equivalence to the standard definition see \cite{MP17}.

 \begin{defin}\label{defn:dualizable}
 A pattern $P$ in a solid torus $V$ is dualizable if and only if $\widehat{P}$ is isotopic to $\widehat{\lambda_V}$ in $S^1 \times S^2$. 
 \end{defin}
 
 Since for a $P$ dualizable pattern $\widehat{P}$ is isotopic to $\widehat{\lambda_V}$ in $S^1\times S^2$ we have that $S^1\times S^2\setminus \nu(\widehat{P})=:V^*$ is a solid torus. By defining $\lambda_{V^*}$ to be the image of $\widehat{\lambda_P}$ in $S^1\times S^2\setminus \nu(\widehat{P})$ we equip $V^*$ with a natural parametrization. As such we can make the following definition.
 
\begin{defin}
Define $P^*$ to be $\widehat{\lambda_V}$ restricted to $S^1\times S^2\setminus \nu(\widehat{P})=V^*$ 
\end{defin}

\begin{figure}
\includegraphics[width=9cm]{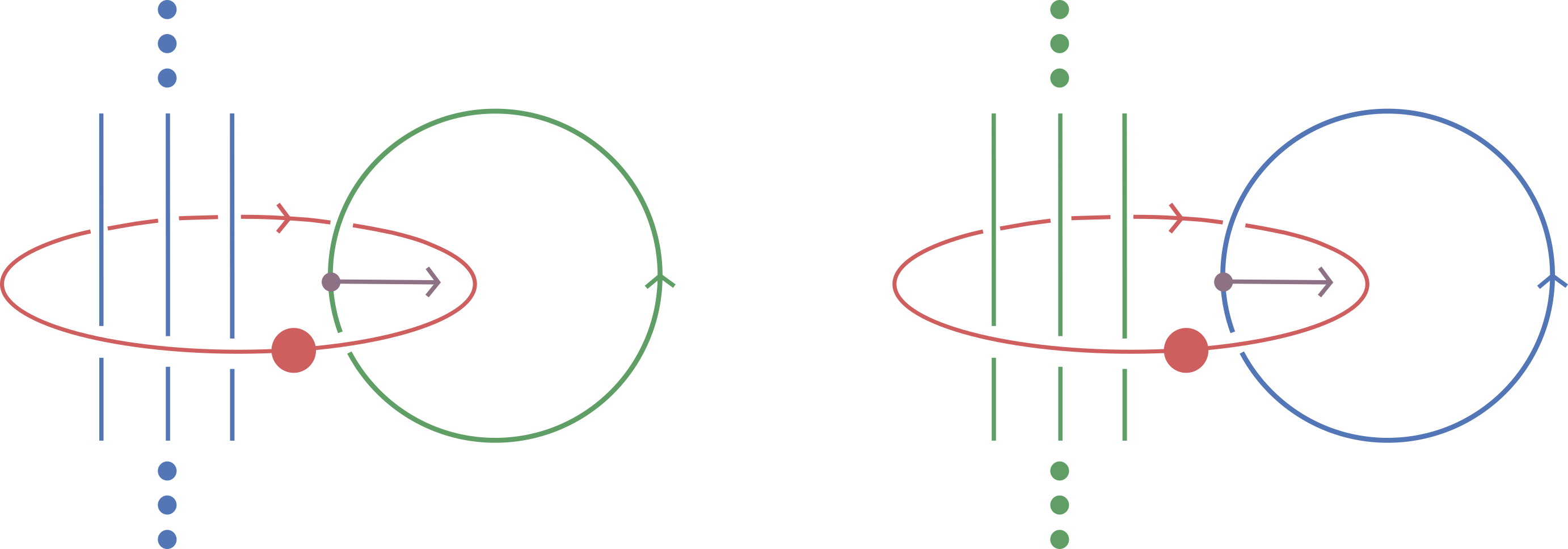}
\caption{}
\label{fig:framingcorrect}
\end{figure}

\vspace{-8pt}
\begin{proof}[Proof of Proposition \ref{prop:dualequiv}]

We begin by proving the first assertion. Suppose $X$ satisfies the hypotheses of Theorem \ref{thm:diffeotraces}, and let $L$ be the decorated link which presents the handle decomposition in the statement of Theorem \ref{thm:diffeotraces}. Let $\ell$ denote the diffeomorphism of $X$ described in the proof of Theorem \ref{thm:diffeotraces} obtained from sliding $B$ over $G$, then canceling $G$ and $R$ to obtain $X_0(K)$, and let $\ell_*$ denote the analogous diffeomorphism used to obtain $X_0(K')$.

We will consider two other natural handle decompositions of $X$, described by decorated links $J$ and $J_*$ respectively. To obtain $J$ from $L$, isotope $L$ so that $R$ has no self crossings in the diagram and so that $G\cup D_R$ is a single point. Then slide $G$ under the one handle (across $R$) as needed until $G$ has no self crossings in the diagram. Let $r$ denote the number of slides this required, counted with sign. Then slide G across $R$ $(-r)$-times as indicated in the left hand side of Figure \ref{fig:framingcorrect}. At this point the framing on $G$ is 0 and $lk(G,B)=0$, but perhaps $B\cap D_G\neq\emptyset$. If this is the case, slide B across $R$ until $B\cap D_G=\emptyset$. Define $J$ to be the decorated link associated to the handle decomposition at this point and define $f$ to be the diffeomorphism of $X$  just described. 

The decorated link $J_*$ and diffeomorphism $f_*$ are defined in the same way, but with the roles of $B$ and $G$ reversed. 

Considering now the link $B\cup G$ in the boundary of the one-handlebody in the handle description $J$, we see that both $B$ and $G$ are isotopic to $S^1\times \{pt\}$ (to see this for $B$ consider its image under $f_{*}\circ f^{-1}$). As such $B\in S^1\times S^2\setminus\nu(G)$ gives a dualizable pattern which we'll call $P$, and $G\in S^1\times S^2\setminus\nu(B)$ it's associated dual pattern $P^*$. Since we can cancel $G$ and $R$ in $J$, we see that $X$ is diffeomorphic to $X_0(P(U))$; let $g$ denote this handle canceling diffeomorphism. Similarly let $g_*$ denote the canceling diffeomorphism from $J_*$ to $X_0(P^*(U))$. 

Now we claim that $K\simeq P(U)$. To see this, observe that the orientation preserving diffeomorphism $g\circ f\circ \ell^{-1}|_\partial:S^3_0(K)\to S^3_0(P(U))$ sends the surgery dual for $K$ to the surgery dual for $P(U)$, preserving the framings corresponding to the respective $S^3$ surgeries. By performing these $S^3$ surgeries, $g\circ f\circ \ell^{-1}|_\partial$ yields a map $\Phi:S^3\to S^3$ taking $K$ to $P(U)$, hence $K\simeq P(U)$. That $K'\simeq P^*(U)$ follows similarly. 

The proof of the second statement is similar. Given a dualizable pattern $P$ with dual pattern $P^*$, one can write down decorated link $J$ and diffeomorphism $g$ as in the proof of the first statement. From $J$ we can obtain another handle decomposition of $X:=X_0(P(U))$; slide $B$ across $R$ until $B$ is isotopic to $S^1\times\{pt\}$. If this sequence of slides required $r$ slides, counted with sign, then perform another $-r$ slides of $B$ over $R$, as in the right hand side of Figure \ref{fig:framingcorrect}. Define $L$ to be the decorated link associated to the handle decomposition at this point and define $h$ to be the diffeomorphism of $X$ just described. 

Observe that $L$ satisfies the hypothesis of Theorem \ref{thm:diffeotraces}, and as before let $\ell$ denote the diffeomorphism of $X$ obtained from sliding $B$ over $G$, then canceling $G$ and $R$ to obtain $X_0(K)$, and let $\ell_*$ denote the analogous diffeomorphism used to obtain $X_0(K')$. 

From $L$ we use the diffeomorphism $f_*$ from before to define another handle diagram $J_*$ of $X$. By the definition of $P^*$ the diffeomorphism $g_*$ from before gives a diffeomorphism from $X$ to $X_0(P*(U))$. We then check that $K\simeq P(U)$ and $K'\simeq P^*(U)$ as in the proof of the first statement.

\end{proof}
\end{document}